\documentclass[11pt]{article}

\usepackage[english]{babel}
\usepackage{csquotes}

\usepackage{amssymb}
\usepackage{amsthm}
\usepackage{amsmath}
\usepackage{mathrsfs}

\usepackage[a4paper, margin=3cm]{geometry}
\usepackage{color}
\usepackage{hyperref}
\usepackage{numprint}

\usepackage{enumerate}

\theoremstyle{plain}
\newtheorem{theorem}{Theorem}[section]

\newtheorem{lemma}[theorem]{Lemma}
\newtheorem{proposition}[theorem]{Proposition}
\newtheorem{corollary}[theorem]{Corollary}

\theoremstyle{definition}

\theoremstyle{remark}
\newtheorem{remark}[theorem]{Remark}

\numberwithin{equation}{section}

\newcommand{\longto}{\longrightarrow}

\usepackage{tikz}
\usetikzlibrary{matrix,arrows}

\newcommand{\ag}{\mathscr{A}_{g}}
\newcommand{\agn}{\mathscr{A}_{g,N}}
\newcommand{\agsat}{\mathscr{A}_{g}^{\text{Sat}}}

\newcommand{\agnsat}{\mathscr{A}_{g,N}^{\text{Sat}}}
\newcommand{\agntor}{\mathscr{A}_{g,N}^{\text{tor}}}
\newcommand{\amod}{A\text{-{\bf Mod}}}

\newcommand{\CC}{\mathbb{C}}
\newcommand{\RR}{\mathbb{R}}

\newcommand{\EE}{\mathbb{E}}
\newcommand{\FF}{\mathbb{F}}
\newcommand{\NN}{\mathbb{N}}
\newcommand{\PP}{\mathbb{P}}

\newcommand{\HH}{\mathcal{H}}

\newcommand{\esat}{\EE_{\text{Sat}}}
\newcommand{\etor}{\EE_{\text{tor}}}

\newcommand{\fp}{\mathbb{F}_p}
\newcommand{\fpbar}{\overline{\mathbb{F}}_p}

\newcommand{\GL}{\operatorname{GL}}
\newcommand{\GSp}{\operatorname{GSp}}

\renewcommand{\H}{\operatorname{H}}

\newcommand{\into}{\hookrightarrow}
\newcommand{\invlim}{\varprojlim}

\newcommand{\oh}{\mathscr{O}}
\newcommand{\osat}{\omega_{\text{Sat}}}
\newcommand{\otor}{\omega_{\text{tor}}}
\newcommand{\osatk}{\osat^{\otimes k}}
\newcommand{\otork}{\otor^{\otimes k}}

\newcommand{\phisat}{\Phi_{\text{Sat}}}
\newcommand{\phitor}{\Phi_{\text{tor}}}
\newcommand{\Proj}{\operatorname{Proj}}
\newcommand{\R}{\operatorname{R}}

\newcommand{\Spec}{\operatorname{Spec}}
\newcommand{\Sym}{\operatorname{Sym}}

\newcommand{\zmod}{\ZZ\text{-{\bf Mod}}}
\newcommand{\zpmod}{\ZZ_p\text{-{\bf Mod}}}
\newcommand{\zpgmod}{\ZZ_p[G(g,\ell)]\text{-{\bf Mod}}}
\newcommand{\ZZ}{\mathbb{Z}}

\title{Cohomological vanishing on Siegel modular varieties and applications to lifting Siegel
  modular forms\footnote{We thank Dawei Chen, Nora Ganter, Sam Grushevsky,
    Martin M\"oller, Paul Sobaje and Beno\^\i{}t Stroh for useful discussions.}}
\author{
Alexandru Ghitza\footnote{The first author was supported by 
Discovery Grant DP120101942 from the Australian Research Council and an Early Career
Researcher Grant from the University of Melbourne.}\\
University of Melbourne \\
{\tt aghitza@alum.mit.edu} \\
\and
Scott Mullane\footnote{The second author was partially supported by
an Australian Postgraduate Award.}\\
  Boston College\\
{\tt scott.mullane@bc.edu} \\
}
\date{\today}

\begin{document}
\thispagestyle{empty}

\maketitle

\begin{abstract}
  We use vanishing results for sheaf cohomology on Siegel modular
  varieties to study two lifting problems:
  \begin{enumerate}[(a)]
    \item When can Siegel modular forms (mod $p$) be lifted to characteristic
      zero?  This uses and extends previous results for \emph{cusp} forms by Stroh
      and Lan-Suh.
    \item When is the restriction of Siegel modular forms to the boundary of
      the moduli space a surjective map?  We investigate this question in
      arbitrary characteristic, generalising analytic results of Weissauer and
      Arakawa.
  \end{enumerate}
\end{abstract}

\section{Introduction}

A venerable technique in arithmetic geometry is to take an object defined over
the integers and study its reductions modulo various primes.  In the case of
classical modular forms with algebraic integer coefficients, this reduction
process gives rise to \emph{Serre-type modular forms (mod $p$)}, whose
$q$-expansion coefficients are in $\fpbar$.  There is a more intrinsic way of
producing modular forms with coefficients in $\fpbar$: in the moduli-theoretic
definition of modular forms, consider the moduli space of elliptic curves over
$\fpbar$.  This gives rise to \emph{Katz-type modular forms (mod $p$)}.  The
natural question is whether the two definitions agree--we formulate this
question as follows: ``Do all (Katz-type) modular forms (mod $p$) lift to
characteristic zero?''

\cite[Theorem 1.7.1]{katz-antwerp} provides a partial positive answer:
\begin{theorem}[Katz]
  All modular forms (mod $p$) of weight $k\geq 2$ and level $\Gamma(N)$ with
  $N\geq 3$ lift to characteristic zero.
\end{theorem}
(We give a variant of Katz's argument in the proof of
Theorem~\ref{thm:vanishing_satake}.)

In the context of his computational exploration of Serre's conjecture, Mestre
found examples of modular forms (mod $2$) of weight $1$ that do not lift to
characteristic zero.  His smallest example appears in level $1429$.  These
computations were reproduced and extended by Wiese; both his and Mestre's
approach appear as appendices to~\cite{edixhoven-integral}.  The interested
reader would also benefit from reading Buzzard's note on computing weight $1$
forms~\cite{buzzard}, as well as the novel and systematic approach given in
Schaeffer's PhD thesis~\cite{schaeffer}.

The question of liftability of modular forms (mod $p$) on \emph{higher rank
groups} has recently received some attention.  Stroh proved that scalar-valued
Siegel cusp forms (mod $p$) of degree $2$ and weight $k\geq 4$ for $p>2$, or
degree $3$ and weight $k\geq 5$ for $p>5$, can be lifted to characteristic
zero~\cite[Th\'eor\`eme 1.1]{stroh-lifting}.  More recently, Lan and Suh
proved that on a Shimura variety $X$ of PEL type, any cusp form (mod $p$) for
$p\geq\dim(X)$ of strictly positive parallel cohomological weight lifts to
characteristic zero~\cite[Theorem 4.1]{lan-suh}.  Restricted to the Siegel
case, this gives liftability of scalar-valued Siegel cusp forms of degree $g$
and weight $k\geq g+2$ for $p\geq g(g+1)/2$.


Our main results are Theorem~\ref{thm:vanishing_satake} and
Corollary~\ref{cor:lifting}, which extend these liftability theorems from
Siegel \emph{cusp} forms to Siegel \emph{modular} forms.  The arguments of
Stroh and Lan-Suh are based on vanishing theorems for the cohomology of line
bundles of cusp forms on a \emph{toroidal compactification} of the Siegel
modular variety.  It is then necessary to investigate what happens along the
boundary.  Our strategy is to pass to the \emph{Satake compactification},
whose boundary consists of strata isomorphic to Siegel modular varieties of
\emph{smaller degree}.  In other words, these correspond to smaller instances
of the problem, enabling us to set up an inductive argument.  The missing
ingredient is a comparison between higher cohomology of line bundles on the
toroidal and Satake compactifications.  This is our
Theorem~\ref{thm:koecher_higher}, which uses the vanishing of relative
cohomology of cuspidal forms, proved recently and independently by
Stroh~\cite{stroh-relative} and
Andreatta-Iovita-Pilloni~\cite{andreatta-iovita-pilloni}.

A pleasant side effect of our strategy is that it also yields information on
the surjectivity of the Siegel $\Phi$-operator, which restricts Siegel modular
forms to the boundary of the moduli space.  We give these results in
Section~\ref{sect:surjectivity}; in characteristic zero, we can even handle
vector-valued forms, via a vanishing theorem for vector bundles on Siegel
modular varieties described in Section~\ref{sect:vector}.

There are no known examples of Siegel modular forms (mod $p$) of small weight
that do not lift to characteristic zero.  The naive search for such forms
would require computing with Siegel modular forms of high level; however, this
appears to be presently out of reach even if we restrict to the simplest
setting of scalar-valued forms of degree $2$.



Note that we assume $N\geq 3$ for most of the paper.  In
Section~\ref{sect:lowlevel} we describe how to extend our results to the low
level cases $N=1$ and $N=2$.

\section{Siegel modular varieties and forms}


Let $\agn$ denote the moduli space\footnote{Recall that, by assumption, $N\geq
3$.} of principally polarized $g$-dimensional abelian varieties with full
level $N$ structure.  It is a smooth quasi-projective scheme of dimension
$g(g+1)/2$ over $\ZZ[1/N]$, see~\cite[Theorem 7.9]{git}.  Let $A$ denote the
universal abelian variety, so that $f\colon A\longto\agn$ is a smooth morphism
of relative dimension $g$.

The \emph{Hodge bundle} $\mathbb{E}$ is the rank $g$ vector bundle on $\agn$
defined by
\begin{equation*}
  \mathbb{E}=f_* \Omega^1_{A/\agn}.
\end{equation*}

Let $\rho$ be an irreducible representation of the algebraic group $\GL_g$ and
let $\lambda=(\lambda_1,\ldots,\lambda_g)$ be its highest weight vector.  By
applying $\rho$ to the transition functions of the vector bundle $\mathbb{E}$,
we obtain a rank $d=\dim\rho$ vector bundle $\mathbb{E}^{\rho}$.  An important
special case is $\rho=\det$, and we denote the resulting line bundle by
$\omega=\mathbb{E}^{\det}=\det\mathbb{E}$.

Given a $\ZZ[1/N]$-algebra $B$, the space of \emph{Siegel modular forms} of
degree $g$, weight $\rho$ and level $\Gamma(N)$ with coefficients in $B$ is
the $B$-module defined by
\begin{equation*}
  M_\rho(N;B)=\H^0\left(\agn, \mathbb{E}^\rho\otimes B\right).
\end{equation*}
In particular, if $p$ is a prime number not dividing $N$, the elements of
$M_\rho(N;\fpbar)$ are called \emph{Siegel modular forms (mod $p$)}.

A construction due to Ash-Mumford-Rapoport-Tai associates to a choice of
combinatorial data (a cone decomposition) a \emph{toroidal compactification}
$\agntor$ of $\agn$.  It is possible to choose $\agntor$ in such a way that it
is a smooth projective scheme over $\ZZ[1/N]$, containing $\agn$ as a dense
open subscheme, and such that the boundary divisor $\agntor-\agn$ is simple
with normal crossings.  Moreover, $\agntor\times \Spec \CC$ is a smooth
projective complex manifold when $N\geq 3$.  There is a canonical extension of
the Hodge bundle $\mathbb{E}$ to a rank $g$ vector bundle $\etor$ on
$\agntor$.  The line bundle $\otor=\det\etor$ is the canonical extension of
$\omega$ to $\agntor$.

The \emph{Satake compactification} $\agnsat$ is the normal, proper scheme over
$\ZZ[1/N]$ given by
\begin{equation*}
  \agnsat=\Proj\left(\bigoplus_{k\geq 0} \H^0\left(\agntor, \otork\right)\right).
\end{equation*}
The main properties of $\agnsat$ are given in~\cite[Theorem
V.2.5]{faltings-chai}.  It contains $\agn$ as a dense open subscheme, and the
line bundle $\omega$ on $\agn$ extends to an ample line bundle $\osat$ on
$\agnsat$.  There is a canonical extension of the Hodge bundle $\mathbb{E}$ to
a coherent sheaf (but not a vector bundle) $\esat$ on $\agnsat$; similarly,
any twist of the Hodge bundle $\mathbb{E}^\rho$ on $\agn$ extends canonically
to a coherent sheaf $\esat^\rho$ on $\agnsat$.

The K\"ocher principle states that, if $g>1$, a Siegel modular form with
coefficients in some $\ZZ[1/N]$-algebra $B$ extends
uniquely to the Satake and toroidal compactifications:
\begin{align*}
  \H^0\left(\agnsat,\esat^\rho\otimes B\right)
  &=\H^0\left(\agn,\mathbb{E}^\rho\otimes B\right);\\
  \H^0\left(\agntor,\etor^\rho\otimes B\right)
  &=\H^0\left(\agn,\mathbb{E}^\rho\otimes B\right).
\end{align*}
The Satake case can be found in~\cite[Proposition 5]{ghitza-cuspidal}, and the
toroidal case in~\cite[Proposition V.1.8]{faltings-chai}.

If $D=\agntor-\agn$ denotes the boundary divisor of a toroidal
compactification $\agntor$, we define the space of \emph{Siegel cusp forms} of
weight $\rho$ and level $\Gamma(N)$ with coefficients in $B$ to be
\begin{equation*}
  S_\rho(N;B)=\H^0\left(\agntor,\etor^\rho(-D)\otimes B\right).
\end{equation*}
In other words, these are the Siegel modular forms that vanish along the
boundary of $\agntor$.  Their definition is independent of the choice of
toroidal compactification~\cite[page 144]{faltings-chai}.  
It is also possible to define cusp forms using the Satake compactification. If
$\Delta=\agnsat-\agn$ denotes the boundary of the Satake compactification and
$\mathscr{S}_\rho$ is the sheaf kernel of the restriction
$\esat^\rho\longto\esat^\rho\Big|_{\Delta}$.  By~\cite[Proposition
7]{ghitza-cuspidal}, the global sections of $\mathscr{S}_\rho$ are precisely
the cusp forms of weight $\rho$ and we have
\begin{equation*}
  S_\rho(N;B)=\H^0\left(\agnsat,\mathscr{S}_\rho\otimes B\right).
\end{equation*}

\section{Review of cohomology and base change}

Siegel modular forms are global sections of certain vector bundles.  The issue
of their liftability from positive characteristic to characteristic zero is
thus a question of certain cohomology groups commuting with base change.  This
is an essential topic in algebraic geometry, treated in many of the standard
references.  However, to our knowledge, none of these references contains a
precise statement of the result we need (Corollary~\ref{cor:cohom_base_change}).  We
give a proof in this section, which does little more than piece together the
necessary ingredients from~\cite[Chapter 7]{EGA3} and~\cite[Chapter
III]{hartshorne-algebraic-geometry}.  Another useful treatment of these
questions can be found in~\cite[Chapter 28]{vakil-foag}.

\begin{theorem}[{\cite[Corollary 7.5.5]{EGA3}}]
  \label{thm:ega_base}
  Let $A$ be a local noetherian ring with residue field $k=A/\mathfrak{m}$.
  Let $T_\bullet$ be a homological functor $\amod\to\zmod$ that commutes with
  direct limits.  Suppose that for every $q$ and every finitely generated
  $A$-module $M$, $T_q(M)$ is finitely generated and the canonical
  homomorphism
  \begin{equation*}
    \widehat{T_q(M)}\longto\invlim_{n} T_q\left(M\otimes_A
    A/\mathfrak{m}^{n+1}\right)
  \end{equation*}
  is bijective.
  \begin{enumerate}[(i)]
    \item If $T_q(k)=0$ then $T_q(M)=0$ for any $A$-module $M$, $T_{q+1}$ is
      right exact and $T_{q-1}$ is left exact.
    \item If $T_{q-1}(k)=T_{q+1}(k)=0$ then $T_q$ is exact, the canonical
      homomorphism
      \begin{equation*}
        T_q(A)\otimes_A M\longto T_q(M)
      \end{equation*}
      is bijective and $T_q(A)$ is a free $A$-module.
  \end{enumerate}
\end{theorem}

\begin{corollary}
  \label{cor:cohom_base_change}
  Let $A$ be a local noetherian ring with residue field $k=A/\mathfrak{m}$.
  Let $X$ be a projective scheme over $\Spec A$ and $\mathcal{F}$ a coherent
  $\oh_X$-module, flat over $\Spec A$.
  \begin{enumerate}[(i)]
    \item If $\H^q(X\times\Spec k,\mathcal{F})=0$, then
      $\H^q(X,\mathcal{F}\otimes_A M)=0$ for any $A$-module $M$.
    \item If
      \begin{equation*}
        \H^{q-1}(X\times\Spec k,\mathcal{F})=
        \H^{q+1}(X\times\Spec k,\mathcal{F})=0,
      \end{equation*}
      then the canonical map
      \begin{equation*}
        \H^q(X,\mathcal{F})\otimes_A M\longto
        \H^q(X,\mathcal{F}\otimes_A M)
      \end{equation*}
      is bijective and $\H^q(X,\mathcal{F})$ is a free $A$-module.
  \end{enumerate}
\end{corollary}
\begin{proof}
It suffices to show that $T_q(M):=\H^q(X,\mathcal{F}\otimes_A M)$
satisfies the conditions of Theorem~\ref{thm:ega_base}.

By~\cite[Proposition III.12.1]{hartshorne-algebraic-geometry}, $T_\bullet$ is
a homological functor.  It commutes with direct limits by~\cite[Proposition
III.2.9]{hartshorne-algebraic-geometry}.

By~\cite[Proposition III.12.2]{hartshorne-algebraic-geometry}, there exists a
complex $L^\bullet$ of finitely generated free $A$-modules and an isomorphism
of functors (in $M$)
\begin{equation*}
  T_q(M)\cong\H^q(L^\bullet\otimes_A M).
\end{equation*}
In particular, the $A$-modules $L^j$ are flat, so we are in the setting
of~\cite[Section 7.4]{EGA3}.  As the $L^j$ are also finitely generated,
\cite[Proposition 7.4.7]{EGA3} indicates that $T_q(M)$ is finitely generated
for any finitely generated $M$, and that the canonical map
\begin{equation*}
  \widehat{T_q(M)}\longto\invlim_{n} T_q\left(M\otimes_A
  A/\mathfrak{m}^{n+1}\right)
\end{equation*}
is an isomorphism.
\end{proof}

\section{Positivity of Hodge line bundles}

We fix a choice of smooth, projective toroidal compactification $\agntor$ of
$\agn$, we let $D$ denote the simple normal crossings divisor $\agntor-\agn$,
and we let $\etor$ denote the canonical extension to $\agntor$ of the Hodge
bundle $\mathbb{E}$ on $\agn$.  The objective of this section is to collect
results about the positivity properties of $\otor=\det\etor$ and deduce the
vanishing of certain cohomology groups.

A line bundle $\mathcal{L}$ on a projective variety $X$ is
\emph{ample} if
\begin{equation*}
  (\mathcal{L}\cdot C)_X > 0
\end{equation*}
for every closed reduced irreducible curve $C\subset X$; here the
intersection number 
$(\mathcal{L}\cdot C)_X$ is the coefficient of $m$ in the polynomial
$\chi(\oh_C\otimes\mathcal{L}^{\otimes m})$, where $\chi(\mathcal{F})$ denotes
the Euler characteristic of the sheaf $\mathcal{F}$.  Our interest in ampleness is
motivated by the following classical result (see~\cite[Theorem
4.2.1]{lazarsfeld-positivity-1}):
\begin{theorem}[Kodaira Vanishing Theorem]
  Let $X$ be a smooth complex projective variety and
  $\mathcal{L}$ an ample line bundle on $X$.  For all $q>0$ we have
  \begin{equation*}
    \H^q\left(X, \omega_X\otimes\mathcal{L}\right)=0,
  \end{equation*}
  where $\omega_X$ is the canonical sheaf of $X$.
\end{theorem}

Unfortunately, $\otor$ is generally not an ample line bundle.  Luckily, it
comes close enough that we can still deduce vanishing of cohomology, as we now
see.

A line bundle $\mathcal{L}$ on a projective variety $X$ is
\emph{numerically effective} (\emph{nef}) if
\begin{equation*}
  (\mathcal{L}\cdot C)_X \geq 0
\end{equation*}
for every closed reduced irreducible curve $C\subset X$.

A line bundle $\mathcal{L}$ on an $n$-dimensional projective variety $X$ is
\emph{big} if there is a constant $c>0$ such that
\begin{equation*}
  \dim\H^0\left(X, \mathcal{L}^{\otimes m}\right)\geq c\cdot m^n
\end{equation*}
for sufficiently large $m\in\NN$.

For the purposes of vanishing of higher cohomology, we can replace
\emph{ample} with \emph{nef and big} (see~\cite[Theorem
4.3.1]{lazarsfeld-positivity-1}):
\begin{theorem}[Kawamata-Viehweg Vanishing Theorem]
  \label{thm:kawamata-viehweg}
  Let $X$ be a smooth complex projective variety and $\mathcal{L}$ a nef and
  big line bundle on $X$.  For all $q>0$ we have
  \begin{equation*}
    \H^q\left(X, \omega_X\otimes\mathcal{L}\right)=0,
  \end{equation*}
  where $\omega_X$ is the canonical sheaf of $X$.
\end{theorem}

\begin{proposition}
  \label{prop:otor_nef_big}
  The line bundle $\otor$ on $\agntor$ is nef and big.
\end{proposition}
\begin{proof}
  The line bundle $\osat$ on $\agnsat$ is ample~\cite[Theorem
  V.2.5(1)]{faltings-chai}.  There is a canonical 
  morphism~\cite[Theorem V.2.5(2)]{faltings-chai}
  \begin{equation*}
    \pi\colon\agntor\longto\agnsat
  \end{equation*}
  obtained as the normalization of the blow-up of $\agnsat$ along a certain ideal
  sheaf~\cite[Theorem V.5.8]{faltings-chai}.  Therefore $\pi$ is surjective,
  proper and birational.

  Since $\osat$ is ample, it is nef and big.  Since pullbacks of nef line
  bundles along proper morphisms are nef~\cite[Example
  1.4.4(1)]{lazarsfeld-positivity-1}, we know that $\otor=\pi^*\osat$ is a nef
  line bundle on $\agntor$.  Similarly, since pullbacks of big line bundles
  along birational morphisms are big~\cite[Section 4.5]{kollar-mori}, we know
  that $\otor=\pi^*\osat$ is a big line bundle on $\agntor$.
\end{proof}

\begin{theorem}
  \label{thm:vanishing}
  Suppose the characteristic of the base field $\FF$ is zero, or positive
  $\geq g(g+1)/2$ and not dividing the level $N\geq 3$.
  If $k\geq g+2$, then
  \begin{equation*}
    \H^q\left(\agntor\times\Spec\FF, (\otor)^{\otimes k}(-D)\right)=0\qquad\text{for all }q>0.
  \end{equation*}
\end{theorem}
\begin{proof}
  In characteristic zero, the vanishing follows from~\ref{prop:otor_nef_big}
  and the Kawamata-Viehweg Vanishing Theorem~\ref{thm:kawamata-viehweg},
  seeing as the canonical sheaf of $\agntor$ is $\otor^{\otimes g+1}(-D)$.
  
  In positive characteristic, the vanishing theorems of Kodaira and
  Kawamata-Viehweg do not hold in general.  However, for $\otor$ on $\agntor$, the
  vanishing is a special case of~\cite[Theorem 4.1]{lan-suh}.
\end{proof}

\section{Analysis of the boundary on toroidal and Satake compactifications}

Let $D$ denote the boundary divisor of a toroidal compactification $\agntor$;
let $\Delta$ denote the boundary of the Satake compactification $\agnsat$.  It
follows from~\cite[Theorem V.2.7]{faltings-chai} that $D$ is the
scheme-theoretic preimage of $\Delta$ under the morphism $\pi$; in other
words, the following is a fibre diagram:

  \begin{center}
  \begin{tikzpicture}[description/.style={fill=white, inner sep=2pt}]
    \matrix (m) [matrix of math nodes, row sep=3em, column sep=3em,
    text height=1.5ex, text depth=0.5ex]
    { D & {\agntor}\\
    {\Delta} & {\agnsat}\\};
    \path[->, font=\scriptsize]
    (m-1-1) edge node[auto] {$i$} (m-1-2); 
    \path[->, font=\scriptsize]
    (m-2-1) edge node[auto] {$j$} (m-2-2);
    \path[->, font=\scriptsize]
    (m-1-2) edge node[auto] {$\pi$} (m-2-2);
    \path[->, font=\scriptsize]
    (m-1-1) edge node[auto] {$\pi|_D$} (m-2-1);
  \end{tikzpicture}
  \end{center}

The following relative vanishing result was proved independently by
Andreatta-Iovita-Pilloni~(\cite[Proposition
8.2.2.4]{andreatta-iovita-pilloni}) and Stroh~(\cite[Th\'eor\`eme
1]{stroh-relative}).  Stroh's proof is very short and more general, as it uses
fewer specific properties of Siegel modular varieties, but only works if the
characteristic is at least $g(g+1)/2$.  The proof in Andreatta-Iovita-Pilloni
works in arbitrary characteristic, but it is based on a more intricate analysis of
the behaviour of the morphism $\pi$ at the boundary.
\begin{theorem}[Andreatta-Iovita-Pilloni, Stroh]
  \label{thm:vanishing_relative}
  Let $\FF$ be a field of arbitrary characteristic.  For all $q>0$ we have
  \begin{equation*}
    \R^q \pi_* \oh_{\agntor}(-D)=0\qquad
    \text{as a sheaf on }\agnsat\times\Spec\FF.
  \end{equation*}
\end{theorem}

\begin{theorem}
  \label{thm:ideal_pushforward}
  If the characteristic of the base field is zero, or positive not dividing
  the level $N$, then 
  \begin{equation*}
    \pi_*\mathcal{I}_D\cong\mathcal{I}_\Delta.
  \end{equation*}
\end{theorem}
\begin{proof}
For some $m\in\NN$, the invertible sheaf $\otor^{\otimes m}$ is generated by
its global sections~\cite[Proposition V.2.1]{faltings-chai}.  This gives a
proper morphism $\agntor\to\PP^n$ into some projective space.  The Stein
factorisation of this morphism (see~\cite[Section III.4.3]{EGA3}) gives
\begin{center}
  \begin{tikzpicture}[description/.style={fill=white, inner sep=2pt}]
    \matrix (m) [matrix of math nodes, row sep=3em, column sep=3em,
    text height=1.5ex, text depth=0.5ex]
    { {\agntor} & {\PP^n}\\
    {\agnsat}\\};
    \path[->, font=\scriptsize]
    (m-1-1) edge node[auto] {$\pi$} (m-2-1);
    \path[->, font=\scriptsize]
    (m-1-1) edge node[auto] {} (m-1-2);
    \path[->, font=\scriptsize]
    (m-2-1) edge node[auto] {} (m-1-2);
  \end{tikzpicture}
\end{center}
which defines both $\agnsat$ and the proper morphism $\pi$.  Consider the
enlarged diagram
\begin{center}
  \begin{tikzpicture}[description/.style={fill=white, inner sep=2pt}]
    \matrix (m) [matrix of math nodes, row sep=3em, column sep=3em,
    text height=1.5ex, text depth=0.5ex]
    { D & {\agntor} & {\PP^n}\\
    {\Delta} & {\agnsat}\\};
    \path[->, font=\scriptsize]
    (m-1-1) edge node[auto] {$i$} (m-1-2); 
    \path[->, font=\scriptsize]
    (m-2-1) edge node[auto] {$j$} (m-2-2);
    \path[->, font=\scriptsize]
    (m-1-2) edge node[auto] {$\pi$} (m-2-2);
    \path[->, font=\scriptsize]
    (m-1-1) edge node[auto] {$\pi|_D$} (m-2-1);
    \path[->, font=\scriptsize]
    (m-1-2) edge node[auto] {} (m-1-3);
    \path[->, font=\scriptsize]
    (m-2-2) edge node[auto] {} (m-1-3);
  \end{tikzpicture}
\end{center}
where the right square is a fibre diagram.  The morphism $\pi|_D$ is the
base change of the proper morphism $\pi$, hence it is proper~\cite[Corollary
II.4.8]{hartshorne-algebraic-geometry}. The two horizontal maps $i$ and $j$
are closed immersions; in particular $i$ is proper~\cite[Corollary
II.4.8]{hartshorne-algebraic-geometry} and $j$ is finite~\cite[Exercise
II.5.5]{hartshorne-algebraic-geometry}.  Since the composition of proper
morphisms is proper, and the composition of finite morphisms is finite, we can
ignore most of the diagram and focus on the triangle
\begin{center}
  \begin{tikzpicture}[description/.style={fill=white, inner sep=2pt}]
    \matrix (m) [matrix of math nodes, row sep=3em, column sep=3em,
    text height=1.5ex, text depth=0.5ex]
    { D & {\PP^n}\\
    {\Delta}\\};
    \path[->, font=\scriptsize]
    (m-1-1) edge node[auto] {$\pi|_D$} (m-2-1);
    \path[->, font=\scriptsize]
    (m-1-1) edge node[auto] {} (m-1-2);
    \path[->, font=\scriptsize]
    (m-2-1) edge node[auto] {} (m-1-2);
  \end{tikzpicture}
\end{center}

By uniqueness, this is the Stein factorisation of the proper morphism
$D\to\PP^n$ (up to an automorphism of $\PP^n$).  In particular
\begin{equation*}
  \left(\pi|_D\right)_*\oh_D\cong\oh_\Delta.
\end{equation*}

Consider the defining short exact sequence for the ideal sheaf
$\mathcal{I}_D$:
\begin{equation*}
  0\longto \mathcal{I}_D\longto \oh_{\agntor}\longto i_*\oh_D\longto 0.
\end{equation*}
We can take higher direct images $\R^\bullet\pi_*$ to get a long exact sequence
of $\oh_{\agnsat}$-modules starting with
\begin{equation*}
  0\longto \pi_*\mathcal{I}_D\longto \pi_*\oh_{\agntor}\longto
  \pi_* i_*\oh_D\longto \R^1\pi_*\mathcal{I}_D.
\end{equation*}

According to Theorem~\ref{thm:vanishing_relative}, the sheaf $\R^1\pi_*\mathcal{I}_D$
is zero.  We get a diagram of $\oh_{\agnsat}$-modules
\begin{center}
  \begin{tikzpicture}[description/.style={fill=white, inner sep=2pt}]
    \matrix (m) [matrix of math nodes, row sep=3em, column sep=3em,
    text height=1.5ex, text depth=0.5ex]
    { 0 & {\pi_*\mathcal{I}_D} & {\pi_*\oh_{\agntor}} & {\pi_* i_* \oh_D} & 0\\
    0 & {\mathcal{I}_\Delta} & {\oh_{\agnsat}} & {j_*\oh_\Delta} & 0\\};
    \path[->, font=\scriptsize]
    (m-1-1) edge node[auto] {} (m-1-2); 
    \path[->, font=\scriptsize]
    (m-1-2) edge node[auto] {} (m-1-3); 
    \path[->, font=\scriptsize]
    (m-1-3) edge node[auto] {} (m-1-4); 
    \path[->, font=\scriptsize]
    (m-1-4) edge node[auto] {} (m-1-5); 
    \path[->, font=\scriptsize]
    (m-2-1) edge node[auto] {} (m-2-2); 
    \path[->, font=\scriptsize]
    (m-2-2) edge node[auto] {} (m-2-3); 
    \path[->, font=\scriptsize]
    (m-2-3) edge node[auto] {} (m-2-4); 
    \path[->, font=\scriptsize]
    (m-2-4) edge node[auto] {} (m-2-5); 
    \path[->, font=\scriptsize]
    (m-1-2) edge node[auto] {} (m-2-2); 
    \path[->, font=\scriptsize]
    (m-1-3) edge node[auto] {$\cong$} (m-2-3); 
    \path[->, font=\scriptsize]
    (m-1-4) edge node[auto] {$\cong$} (m-2-4); 
  \end{tikzpicture}
\end{center}
where the middle vertical arrow is an isomorphism (by the properties of the
Stein factorisation), and the right vertical arrow is an isomorphism:
\begin{equation*}
  \pi_* i_* \oh_D = j_*\left(\left(\pi|_D\right)_* \oh_D\right)
  \cong j_*\oh_\Delta.
\end{equation*}

We conclude that $\pi_* \mathcal{I}_D\cong\mathcal{I}_\Delta$.
\end{proof}

\begin{lemma}[{\cite[Exercise III.8.1]{hartshorne-algebraic-geometry}}]
  \label{lem:relative}
Suppose $\pi\colon X\to Y$ is a continuous map of topological spaces and
$\mathcal{F}$ is a sheaf of abelian groups on $X$ such that
\begin{equation*}
  \R^q f_*(\mathcal{F})=0\qquad\text{for all }q>0.
\end{equation*}
Then there are natural isomorphisms
\begin{equation*}
  \H^q(X,\mathcal{F})=H^q(Y,f_*\mathcal{F})\qquad\text{for all }q\geq 0.
\end{equation*}
\end{lemma}
(This is a degenerate case of the Leray spectral sequence.)

\begin{theorem}
  \label{thm:koecher_higher}
  Suppose the characteristic of the base field $\FF$ is zero, or positive not dividing
  the level $N\geq 3$.
  For any $k\geq 0$ and any $q\geq 0$ we have
  \begin{equation*}
    \H^q\left(\agntor\times\Spec\FF, \otork(-D)\right) =
    \H^q\left(\agnsat\times\Spec\FF, \osatk\otimes\mathcal{I}_\Delta\right).
  \end{equation*}
\end{theorem}
\begin{proof}
  According to Theorem~\ref{thm:vanishing_relative}, we have
  \begin{equation*}
    \R^q \pi_*\left(\oh_{\agntor}(-D)\right)=0\qquad\text{for all }q>0.
  \end{equation*}

  Using the projection formula~\cite[Proposition 0.12.2.3]{EGA3}, we see that
  \begin{align*}
    \R^q \pi_*\left(\otork(-D)\right)
    &= \R^q \pi_*\left(\oh_{\agntor}(-D)\otimes\pi^*\osatk\right)\\ 
    &= \R^q \pi_*\left(\oh_{\agntor}(-D)\right)\otimes\osatk\\
    &= 0.
  \end{align*}

  We conclude that
  \begin{equation*}
    \H^q\left(\agntor, \otork(-D)\right) =
    \H^q\left(\agnsat, \pi_*\left(\otork(-D)\right)\right) =
    \H^q\left(\agnsat, \osatk\otimes\mathcal{I}_\Delta\right),
  \end{equation*}
  where the first equality comes from Lemma~\ref{lem:relative} and the second
  equality from Theorem~\ref{thm:ideal_pushforward}.
\end{proof}


The following result\footnote{Grothendieck attributes to Serre this \emph{astuce} of
reducing to the case of an irreducible space, see~\cite[page
29]{grothendieck-serre}.} is well-known as part of the proof of Grothendieck's
cohomological dimension theorem, see the original~\cite[Th\'eor\`eme
3.6.5]{grothendieck-tohoku} or the presentation in~\cite[Theorem
III.2.7]{hartshorne-algebraic-geometry}:
\begin{lemma}
  \label{lem:serre_astuce}
  Let $X$ be a topological space with finitely many irreducible components.
  Let $\mathcal{F}$ be a sheaf of abelian groups on $X$, and let $q\geq 0$.
  If $\H^q(Y, \mathcal{F}|_Y)=0$ for all irreducible components $Y$ of $X$,
  then $\H^q(X, \mathcal{F})=0$.
\end{lemma}

\begin{theorem}
  \label{thm:vanishing_satake}
  Suppose the characteristic of the base field $\FF$ is zero, or positive
  $\geq g(g+1)/2$ and not dividing
  the level $N\geq 3$.  For all $k\geq g+2$ and $q > 0$, we have
  \begin{equation*}
    \H^q\left(\agnsat\times\Spec\FF, \osatk\right)=0.
  \end{equation*}
\end{theorem}
\begin{proof}
  We base change our spaces to $\FF$ and omit $\Spec\FF$ from the notation,
  for simplicity.

  We proceed by induction on $g$.

  The base case $g=1$ is well-known but worth including.  Here
  $\agnsat=\agntor$ is the modular curve $X(N)$, so the vanishing is clear for
  $q>1$.  The sheaf $\osat$ has positive degree, and so does the effective
  divisor $D$.  By Serre duality, we have
  \begin{equation*}
    \H^1\left(\agnsat, \osat^{\otimes k}\right)=
    \H^0\left(\agnsat, \osat^{\otimes 2-k}(-D)\right)^\vee.
  \end{equation*}
  But if $k\geq g+2=3$, then $2-k<0$ so $\osat^{\otimes 2-k}(-D)$ has negative
  degree, and hence no nonzero global sections. 

  For the induction step, consider the short exact sequence that defines the
  ideal sheaf $\mathcal{I}_\Delta$:
  \begin{equation*}
    0\longto \mathcal{I}_\Delta\longto \oh_{\agnsat}\longto
    j_*\oh_\Delta\longto 0.
  \end{equation*}
  Tensoring with the line bundle $\osatk$ gives another short exact sequence
  \begin{equation}
    \label{eq:ses}
    0\longto\mathcal{I}_\Delta\otimes\osatk\longto \osatk\longto
    \osatk|_\Delta\longto 0.
  \end{equation}
  Applying Theorems~\ref{thm:koecher_higher}, then~\ref{thm:vanishing}, for
  $k\geq g+2$ we have
  \begin{equation*}
    \H^q\left(\agnsat,\mathcal{I}_\Delta\otimes\osatk\right)
    =\H^q\left(\agntor,\mathcal{I}_D\otimes\otork\right)=0.
  \end{equation*}

  The long exact sequence of cohomology associated with~\eqref{eq:ses}
  has pieces of the form
  \begin{equation*}
    \H^q\left(\agnsat,\mathcal{I}_\Delta\otimes\osatk\right)\longto
    \H^q\left(\agnsat,\osatk\right)\longto
    \H^q\left(\agnsat,\osatk|_\Delta\right).
  \end{equation*}
  If $q>0$, we have just seen that the leftmost group is zero, so it
  suffices to prove that the rightmost group is zero.

  Let $C_1,\ldots,C_b$ be the irreducible components of the boundary $\Delta$.
  Each component $C_j$ is isomorphic to $\mathscr{A}_{g-1,N}^\text{Sat}$, with
  \begin{equation*}
    \omega_{\text{Sat},g}|_{C_j} = \omega_{\text{Sat},g-1}
  \end{equation*}
  (by~\cite[Theorem V.2.5(4)]{faltings-chai}). 
    By the induction hypothesis,
  $\H^q(\mathscr{A}_{g-1,N}^\text{Sat},\omega_{\text{Sat},g-1}^{\otimes
  k})=0$.  So each component $C_j$ of $\Delta$ satisfies
  \begin{equation*}
    \H^q(C_j,\osatk|_{C_j})=0.
  \end{equation*}
  Finally, Lemma~\ref{lem:serre_astuce} allows us to conclude that $\H^q(\agnsat,\osatk|_\Delta)=0$. 
\end{proof}

\begin{corollary}
  \label{cor:lifting}
  Let $N\geq 3$.
  Suppose $p\geq g(g+1)/2$ is a prime not dividing $N$.  For all $k\geq g+2$,
  the base change morphism
    \begin{equation*}
      M_k(\Gamma(N))\otimes\fp\longto M_k(\Gamma(N);\fp)
    \end{equation*}
  is an isomorphism.
\end{corollary}
\begin{proof}
  Over the local Noetherian ring $\ZZ_p$, Theorem~\ref{thm:vanishing_satake} and
  Corollary~\ref{cor:cohom_base_change} imply that the base change morphism is
  an isomorphism.  By flat base change this implies that
  \begin{equation*}
    \H^0\left(\agnsat,\osatk\right)\otimes_{\ZZ[1/N]}\fp
    \longto
    \H^0\left(\agnsat,\osatk\otimes_{\ZZ[1/N]}\fp\right)
  \end{equation*}
  is an isomorphism.
  The result now follows by K\"ocher's principle.
\end{proof}

See Corollary~\ref{cor:small_level} for an extension of this result to the
small levels $N=1,2$.

\section{Vanishing of vector bundles in characteristic zero}
\label{sect:vector}
We start with a variant of a vanishing theorem for cohomology of
vector bundles, due to Demailly.  We follow Manivel's simplified proof of this
result, as presented in~\cite[Section 7.3.B]{lazarsfeld-positivity-2}.

Given a vector bundle $E$ on a projective scheme $X$, let $\PP(E)$ denote the
projective bundle of $E$ parametrising hyperplane sections in the fibres
$E_x$.  We say that $E$ is \emph{nef} over $X$ if $\oh_{\PP(E)}(1)$ is a nef
line bundle over $\PP(E)$.  We refer the reader to~\cite[Section
6.2.B]{lazarsfeld-positivity-2} for basic properties of nef vector bundles,
and to~\cite[Appendix A]{lazarsfeld-positivity-1} for a short summary of
projective bundles.

\begin{theorem}[Demailly-Manivel]\label{thm:demailly}
  Let $X$ be a smooth projective complex variety.

  Let $E$ be a nef vector bundle of rank $e$ on $X$, and let $L$ be a nef and big line
  bundle on $X$.  
  
  Let $\lambda=(\lambda_1\geq\ldots\geq\lambda_e)\in\ZZ^e$,
  $\lambda_e\geq 0$; let $h=h(\lambda)$ denote the
  number of nonzero parts $\lambda_i$ of $\lambda$, and let $E^\lambda$ be the
  vector bundle associated to the irreducible representation of $\GL_e$ with
  highest weight $\lambda$.  
  
  Then
  \begin{equation*}
    \H^q\left(X,\omega_X\otimes E^\lambda\otimes (\det E)^{\otimes
    h}\otimes L\right)=0\qquad\text{for all }q>0.
  \end{equation*}
\end{theorem}
\begin{proof}
  By the definition of $h$ we have 
  $\lambda=(\lambda_1\geq\ldots\geq\lambda_h\geq 0\geq\ldots\geq 0)$.
  Let $m=\lambda_1+\ldots+\lambda_h$, and let $F=E\oplus E\oplus\ldots\oplus
  E$, where we take $h$ summands.  Then $F$ is a nef vector bundle on $X$, and
  $\det F=(\det E)^{\otimes h}$.

  We apply Theorem~\ref{thm:griffiths} to $F$ and get
  \begin{equation}\label{eq:grif}
    \H^q\left(X,\omega_X\otimes(\Sym^m F)\otimes (\det F)\otimes L\right)=0
    \qquad\text{for all }q>0.
  \end{equation}
  But
  \begin{equation*}
    \Sym^m F=\bigoplus_{m_1+\ldots+m_h=m} (\Sym^{m_1} E)\otimes\ldots
    \otimes (\Sym^{m_h} E)
  \end{equation*}
  In particular, we have
  \begin{equation*}
    E^\lambda \subset (\Sym^{\lambda_1} E)\otimes \ldots\otimes
    (\Sym^{\lambda_h} E)\subset \Sym^m F,
  \end{equation*}
  where both inclusions are as direct summands.  Therefore~\eqref{eq:grif}
  gives us
  \begin{equation*}
    \H^q\left(X,\omega_X\otimes E^\lambda\otimes (\det E)^{\otimes h}\otimes
    L\right)=0\qquad\text{for all }q>0.
  \end{equation*}
\end{proof}

For completeness, we give a proof of
the following variant of the Griffiths vanishing theorem, which is stated
in~\cite[Example 7.3.3]{lazarsfeld-positivity-2}.

\begin{theorem}[Griffiths]\label{thm:griffiths}
  Let $X$ be a smooth projective complex variety of dimension $n$.

  Let $F$ be a nef vector bundle of rank $r$ on $X$, and let $L$ be a nef and
  big line bundle on $X$.  Then
  \begin{equation*}
    \H^q \left(X,\omega_X\otimes (\Sym^m F)\otimes (\det F)\otimes L\right)
    =0\qquad\text{for all }q>0, m\geq 0.
  \end{equation*}
\end{theorem}
\begin{proof}
 The cotangent bundle sequence for $\pi\colon \PP(F)\to X$ gives
 $\omega_{\PP(F)}=\omega_{\PP(F)/X}\otimes\pi^*\omega_X$. The relative Euler sequence
\begin{equation*}
0\longrightarrow \Omega_{\PP(F)/X}^1\longrightarrow \pi^*F\otimes \mathcal{O}_{\PP(F)}(-1)\longrightarrow\mathcal{O}_{\PP(F)}\longrightarrow 0
\end{equation*}
gives $\omega_{\PP(F)/X}=\pi^*\det F\otimes \mathcal{O}_{\PP(F)}(-r)$. Then for any $m\geq 0$,
\begin{align*}
\pi_*(\omega_{\PP(F)}\otimes\mathcal{O}_{\PP(F)}(m+r))&=\pi_*(\pi^*(\omega_X\otimes\det F)\otimes \mathcal{O}_{\PP(F)}(m))\\
 &=\omega_X\otimes\det F\otimes\pi_*(\mathcal{O}_{\PP(F)}(m))\\
&= \omega_X\otimes \det F\otimes \Sym^mF.
\end{align*}
Furthermore, when $m\geq0$  we have 
\begin{equation*}
R^{r-1}\pi_*\mathcal{O}_{\PP(F)}(-r-m)= (\Sym^mF)^*\otimes \det F^*
\end{equation*}
and all other direct image sheaves of this type vanish. By the projection formula
\begin{align*}
R^i\pi_*(\omega_{\PP(F)}\otimes\mathcal{O}_{\PP(F)}(m+r)\otimes\pi^*L)&=R^i\pi_*(\mathcal{O}_{\PP(F)}(m)\otimes \pi^*(\omega_X\otimes\det F\otimes L))\\
&=R^i\pi_*(\mathcal{O}_{\PP(F)}(m))\otimes \omega_X\otimes\det F\otimes L,
\end{align*}
so the higher direct image sheaves of
$\omega_{\PP(F)}\otimes\mathcal{O}_{\PP(F)}(m+r)\otimes\pi^* L$ vanish. 
Therefore
\begin{equation*}
H^i(X,\omega_X\otimes \det F\otimes \Sym^mF\otimes L)=
H^i(\PP(F),\omega_{\PP(F)}\otimes\mathcal{O}_{\PP(F)}(m+r)\otimes\pi^* L).
\end{equation*}

Observe that $\mathcal{O}_{\PP(F)}(1)$ is nef (by definition) and $\pi^*L$ is
also nef (as the pullback of a nef line bundle under a proper, surjective
map).  Hence $\mathcal{O}_{\PP(F)}(m+r)\otimes \pi^*L$ is nef. 

To show that $\mathcal{O}_{\PP(F)}(m+r)\otimes \pi^*L$ is big we use the
fact that a nef divisor is big if and only if its top intersection is strictly
positive~\cite[Theorem 2.2.16]{lazarsfeld-positivity-1}.  Since the sum of a
nef divisor and a nef and big divisor is nef and big (a nef divisor that
is not big lies on an extremal ray of the nef cone) it suffices to show
that $\mathcal{O}_{\PP(F)}(1)\otimes \pi^*L$ is nef and big.

We have
\begin{equation*}
\left(c_1(\mathcal{O}_{\PP(F)}(1))+c_1(\pi^*L)\right)^{n+r-1}
=\sum_{i=0}^{n+r-1}\binom{n+r-1}{i}\,
c_1\!\left(\mathcal{O}_{\PP(F)}(1)\right)^i\cdot
c_1\!\left(\pi^*L\right)^{n+r-1-i}.
\end{equation*}
An ample divisor restricted to any subvariety is ample.  As a result of this
in our situation we have $A_1\cdots A_{n+r-1}>0$ for ample $A_i$.  In the
limit this gives $D_1\cdots D_{n+r-1}\geq 0$ for nef $D_i$.  Hence
$c_1(\mathcal{O}_{\PP(F)}(1))^i\cdot c_1(\pi^* L)^{n+r-1-i}\geq 0$.

It remains to exhibit one non-zero term in the sum.  We know that $\deg
(c_1(L)^n)>0$ as $L$ is nef and big. The fibres of $\PP(F)\to X$ are
isomorphic to $\PP^{r-1}$. Now $\mathcal{O}_{\PP(F)}(1)$ restricted to each
fibre is $\mathcal{O}_{\PP^{r-1}}(1)$ hence 
$c_1(\pi^*L)^n\cdot
c_1(\mathcal{O}_{\PP(F)}(1))^{r-1}
=\deg(c_1(L)^n)>0$.

So we have that $\mathcal{O}_{\PP(F)}(m+r)\otimes\pi^*L$ is nef and big and an
application of Kawamata-Viehweg vanishing gives the result we require.
\end{proof}

Our interest in these vanishing results comes from the fact that the vector
bundle $\etor$ on $\agntor$ is nef.  (See the proof
of~\cite[Corollary~3.2]{moeller-viehweg-zuo}, where this fact is credited to
Kawamata.  Beware that $\etor$ is denoted $F^{1,0}$
in~\cite{moeller-viehweg-zuo}.)

To simplify the statement of some of the following results, we define
what we mean by an element of $\ZZ^g$ to be ``sufficiently large'' with
respect to $g$.  Let
\begin{equation*}
  \mu=(\mu_1+k\geq\ldots\geq\mu_{g-1}+k\geq k)\in\ZZ^g\quad
  \text{where }\mu_{g-1}\geq 0.
\end{equation*}
Let $\lambda=(\mu_1\geq\ldots\geq\mu_{g-1}\geq 0)$ and let $h(\lambda)$ be the number
of nonzero $\mu_i$'s.  If $k\geq g+h(\lambda)+2$, we say that $\mu$ is
``sufficiently large''.

\begin{theorem}
  \label{thm:siegel_vanish}
  Let $g\geq 2$.  If $\mu\in\ZZ^g$ is ``sufficiently large'', then
  \begin{equation*}
    \H^q\left(\agntor, \etor^{\mu}(-D)\right)=0\qquad\text{for all }q>0.
  \end{equation*}
\end{theorem}
\begin{proof}
  The canonical bundle of $X=\agntor$
  is~\cite[Section VI.4]{faltings-chai}
  \begin{equation*}
    \omega_X=(\otor)^{\otimes g+1}(-D).
  \end{equation*}

  Let $j=k-g-h(\lambda)-1>0$.  Note that $\etor^\mu=\etor^\lambda\otimes
  (\otor)^{\otimes k}$.

  We apply Theorem~\ref{thm:demailly} with $E=\etor$ and $L=(\otor)^{\otimes
  j}$ and get
  \begin{align*}
    \H^q\left(\agntor, \etor^{\mu}(-D)\right)&=
    \H^q\left(\agntor, \etor^{\lambda}\otimes (\otor)^{\otimes k}(-D)\right)\\
    &=\H^q\left(\agntor,(\otor)^{\otimes g+1}(-D)\otimes \etor^\lambda\otimes
    (\otor)^{\otimes h(\lambda)}\otimes (\otor)^{\otimes j}\right)\\
    &=0.
  \end{align*}
\end{proof}

We record two special cases of interest.  First, note that the case of
highest weight $\mu=(k\geq\ldots\geq k)\in\ZZ^g$ gives precisely the vanishing result
for scalar-valued forms which constitutes the characteristic zero part of
Theorem~\ref{thm:vanishing}.  The second special case is that of symmetric
powers, corresponding to highest weight $\mu=(j+k\geq k\geq\ldots\geq k)\in\ZZ^g$:
\begin{corollary}[Symmetric powers]
  \label{cor:sym_vanish}
  If $j\geq 1$ and $k\geq g+3$, then
  \begin{equation*}
    \H^q\left(\agntor, \Sym^j(\etor)\otimes(\otor)^{\otimes
    k}(-D)\right)=0\qquad\text{for all }q>0.
  \end{equation*}
\end{corollary}

\section{Surjectivity of the Siegel operators}
\label{sect:surjectivity}
Let $\rho$ be an irreducible representation of the algebraic group $\GL_g$
with highest weight vector $\lambda=(\lambda_1,\ldots,\lambda_g)$. If
$\Delta=\agsat-\ag$ then the inclusion $i\colon\Delta\hookrightarrow \agsat$ gives
\begin{equation*}
\H^0\left(\agsat,\esat^\rho\otimes i_*\oh_\Delta\otimes
B\right)=\H^0\left(\mathscr{A}_{g-1}^{\text{Sat}},\esat^{\rho^\prime}\otimes
B\right),
\end{equation*}
where $\rho^\prime$ is the irreducible representation of $\GL_{g-1}$ with
highest weight vector $\lambda^\prime=(\lambda_1,...,\lambda_{g-1})$.

The map that takes a section of $\H^0(\agsat,\esat^\rho\otimes B)$ to its
restriction as a section in
$\H^0(\mathscr{A}_{g-1}^{\text{Sat}},\esat^{\rho^\prime}\otimes B)$ is known as the
Siegel $\Phi$-operator. When $B=\CC$, this operator
\begin{equation*}
\Phi:M_\rho^g(1,\CC)\longrightarrow M_{\rho^\prime}^{g-1}(1,\CC)
\end{equation*}
is realised as
\begin{equation*}
(\Phi f)(\tau^\prime)=\lim_{t\to\infty} f\begin{pmatrix}it&0\\0&\tau^\prime   \end{pmatrix}
\end{equation*}
for $\tau^\prime\in\HH_{g-1}$, $t\in \RR$. 


Weissauer shows~\cite[Korollar zum Satz 8, p. 87]{weissauer} that $\Phi$ is
surjective for even $k\geq g+2$.
In the vector-valued case with $g=2$, $\mu=(j+k,k)$ it was proved by Arakawa to be surjective for $j\geq 1$ and $k\geq 5$ in~\cite[Proposition 1.3]{arakawa}.  
Weissauer and Arakawa's proofs are analytic in nature and involve showing certain integrals representing an averaging process converge. 

The Siegel $\Phi$-operator generalises to higher levels as the restriction of
global sections to the boundary of the Siegel variety.  If
$\Delta=\agnsat-\agn$ denotes the boundary of the Satake compactification,
then the inclusion $i\colon\Delta\hookrightarrow\agnsat$ gives the operator
\begin{equation*}
  \phisat\colon\H^0\left(\agnsat, \esat^\rho\otimes B\right)\longto
  \H^0\left(\agnsat,\esat^\rho\otimes i_*\oh_\Delta\otimes B\right).
\end{equation*}

Similarly, if $D=\agntor-\agn$ denotes the boundary divisor of the toroidal
compactification, then the inclusion $j\colon D\hookrightarrow\agntor$ gives
the operator
\begin{equation*}
  \phitor\colon\H^0\left(\agntor, \etor^\rho\otimes B\right)\longto
  \H^0\left(\agntor,\etor^\rho\otimes j_*\oh_D\otimes B\right).
\end{equation*}

We investigate conditions under which these operators are surjective.

Consider the ideal sheaf of $j\colon D\into\agntor$, defined by the short exact
sequence
\begin{equation*}
  0\longto \mathcal{I}_D \longto \oh_{\agntor} \longto j_* \oh_D
  \longto 0.
\end{equation*}

Since $\etor^\rho$ is locally free, tensoring by
$\etor^\rho$ gives a short exact sequence
\begin{equation}
\label{eqn:ses3}
  0\longto\etor^\rho\otimes\mathcal{I}_D \longto \etor^\rho\longto
  \etor^\rho\Big|_D\longto 0.
\end{equation}
We get a long exact sequence in cohomology that features the toroidal operator
\begin{equation*}
  0\longto S_\mu(N)\longto M_\mu(N)\xrightarrow{\,\phitor\,}
  \H^0\left(D, \etor^{\mu}\big|_D\right)\longto
  \H^1\left(\agntor, \etor^\mu\otimes\mathcal{I}_D\right).
\end{equation*}

The
following is a direct consequence of Theorem~\ref{thm:siegel_vanish}.  (For
the definition of ``sufficiently large'', see the paragraph before
Theorem~\ref{thm:siegel_vanish}.)

\begin{theorem}
  \label{thm:surj_zero}
  Let $N\geq 3$.
  Over a field of characteristic zero, we have
  \begin{enumerate}[(i)]
    \item If $\mu$ is ``sufficiently large'', then $\phitor$ on forms of
      degree $g$ and weight
      $\mu$ is surjective.
    \item If $k\geq g+2$, then $\phitor$ on scalar-valued forms of degree $g$
      and weight $k$ is surjective.
    \item If $j\geq 1$ and $k\geq g+3$, then $\phitor$ on forms of degree $g$
      and weight
      $\Sym^j\otimes\det^{\otimes k}$ is surjective.
  \end{enumerate}
\end{theorem}
Note that part (iii) is the toroidal analogue in level $N\geq 3$ of a result proved by Arakawa in
degree $2$ and level $N=1$ for the Satake compactification, see~\cite[Proposition
1.3]{arakawa}.

In positive characteristic, we can restrict to scalar-valued forms and
appeal to the vanishing theorem of Lan and Suh (Theorem~\ref{thm:vanishing}) to
get:
\begin{theorem}
  \label{thm:surj_pos}
  Let $p\geq g(g+1)/2$ be a prime not dividing the level $N\geq 3$.  If $k\geq g+2$, then
  $\phitor$ on scalar-valued forms (mod $p$) of degree $g$ and weight $k$ is
  surjective.
\end{theorem}

The operator $\phisat$ also fits into a long exact sequence
\begin{equation*}
  0\longto S_k(N)\longto M_k(N)\xrightarrow{\,\phisat\,}
  \H^0\left(\Delta, \osatk\big|_\Delta\right)\longto
  \H^1\left(\agnsat, \osatk\otimes\mathcal{I}_\Delta\right).
\end{equation*}
By appealing to Theorems~\ref{thm:surj_zero}(ii), \ref{thm:surj_pos} and
\ref{thm:koecher_higher}, we obtain
\begin{corollary}\label{cor:phisatp}
  Suppose the characteristic of the base field $\FF$ is zero, or a prime 
  $p\geq g(g+1)/2$ not dividing the level $N\geq 3$.  If $k\geq g+2$, then $\phisat$
  on scalar-valued forms over $\FF$ of degree $g$ and weight $k$ is
  surjective. 
\end{corollary}

In characteristic zero, this gives an algebraic proof for a result analogous
to~\cite[Korollar zum Satz 8, p. 87]{weissauer}, which was obtained by
Weissauer using analytic methods.  See Theorem~\ref{thm:surjphi} for a
version of Theorem~\ref{thm:surj_pos} and Corollary~\ref{cor:phisatp} in
level $1$, and Theorem~\ref{thm:liftC} for a result about lifting forms of
level $1$ to forms of level $N\geq 3$ in higher degree.

\section{Levels $1$ and $2$}
\label{sect:lowlevel}

If $N\geq 3$, then for any $L\geq 1$ the canonical morphism
\begin{equation*}
  \mathscr{A}_{g,LN}\longto\agn
\end{equation*}
is a finite covering with Galois group $G(g,L):=\GSp_{2g}(\ZZ/L\ZZ)$.
Therefore, given any $\ZZ[1/LN]$-algebra $B$, we have
\begin{align*}
  M_\rho(\Gamma(N);B) &= M_\rho(\Gamma(LN);B)^{G(g,L)}\\
  S_\rho(\Gamma(N);B) &= S_\rho(\Gamma(LN);B)^{G(g,L)}
\end{align*}
We use these formulas to \emph{define} Siegel modular forms and cusp forms of
levels $N=1$ and $2$ and coefficients in $B$.  (This is independent of the
choice of $L$ invertible in $B$.)

If $p$ is a prime not dividing the order of $G(g,L)$, then the ``invariants'' functor
$\zpgmod\to\zpmod$ given by $M\mapsto M^{G(g,L)}$
is exact.  By the Chinese Remainder Theorem,
\begin{equation*}
  \# G(g,L) = \prod_{\ell^a\| L,\ell\text{ prime}} \# G(g,\ell^a).
\end{equation*}
On the other hand, it is known that if $\ell$ is prime
\begin{equation*}
  \# G(g,\ell^a) =
  \left(\ell^a-1\right)\ell^{ag^2}\prod_{i=1}^g\left(\ell^{2ia}-1\right),
\end{equation*}
which in particular shows that $\# G(g,\ell)$ divides $\# G(g,\ell^a)$.  We
conclude that, in order to find $L$ such that a particular prime $p$ does not
divide $\# G(g,L)$, it is sufficient to consider prime numbers $\ell=L$.


\begin{proposition}
  Let $g\geq 1$ be an integer.  Let $p$ be a prime $> 2g+1$.
  There exists a prime number $\ell\geq 3$ such that $p$ does not divide
  \begin{equation*}
    \# G(g,\ell) = \left(\ell-1\right)\ell^{g^2}
    \prod_{i=1}^g \left(\ell^{2i}-1\right).
  \end{equation*}
  Moreover, the inequality $p>2g+1$ is sharp.
\end{proposition}
\begin{proof}
  Let $\ell\geq 3$ be a prime primitive root modulo $p$.  (It is known that
  there are infinitely many such $\ell$, see for instance~\cite{martin}.)

  Then the order of $\ell$ in the group $\left(\ZZ/p\ZZ\right)^\times$ is
  exactly $p-1>2g$, in other words $\ell^{2g}\not\equiv 1\pmod{p}$, so $p\nmid
  \left(\ell^{2g}-1\right)$.  The same argument forbids $p$ from dividing any
  of the factors in $\# G(g,\ell)$.

  The claim about the sharpness of the inequality $p>2g+1$ can be stated more
  precisely as follows: if $g\geq 1$ and $p$ is a prime $\leq 2g+1$, then $p$
  divides $\# G(g,\ell)$ for all primes $\ell\geq 3$.  This is easily checked
  for $g=1$.  For general $g$, the formula for $\# G(g,\ell)$ shows that
  \begin{equation*}
    \# G(g,\ell) = \# G(g-1,\ell)\cdot \ell^{2g-1}\cdot
    \left(\ell^{2g}-1\right).
  \end{equation*}
  By Fermat's little theorem, if $p=2g+1$ is prime, then either $\ell=p$ or
  $p$ divides $\ell^{2g}-1$.  A simple induction argument concludes the proof.
\end{proof}

We summarize the content of this section and its relevance to the rest of the
paper:
\begin{corollary}
  \label{cor:small_level}
  Let $g\geq 1$ and $p$ be a prime $>2g+1$.  There exists a prime number
  $\ell\geq 3$ such that the functor $\zpgmod\to\zpmod$, $M\mapsto
  M^{G(g,\ell)}$, is exact.
  
  In particular, the statement of Corollary~\ref{cor:lifting} can be extended
  as follows: let $g\geq 1$ and $p$ be a prime $>2g+1$ not dividing $N$.
  For all $k\geq g+2$ the base change morphisms
    \begin{align*}
      M_k(\Gamma(N))\otimes\fp &\longto M_k(\Gamma(N);\fp)\\
      S_k(\Gamma(N))\otimes\fp &\longto S_k(\Gamma(N);\fp)
    \end{align*}
  are isomorphisms.
\end{corollary}

\begin{remark}
The alert reader will have noticed a gap between the condition $p\geq (g+1)g/2$ from
Corollary~\ref{cor:lifting} and the condition $p>2g+1$ assumed in
Corollary~\ref{cor:small_level}.  More precisely, the cases:
\begin{enumerate}[(a)]
  \item $g=1$, $p=2,3$
  \item $g=2$, $p=3,5$
  \item $g=3$, $p=7$
\end{enumerate}
are not covered by Corollary~\ref{cor:small_level}.  For (a), which is classical,
see~\cite[Lemma 1.9]{edixhoven-serre}.  Cases (b) and (c) can presumably
be studied in a similar way, using explicit presentations over $\ZZ$ of the ring of
scalar-valued Siegel modular forms of degree $2$, respectively $3$.
\end{remark}


It is natural to ask whether small level versions of Theorem~\ref{thm:surj_pos} and
Corollary~\ref{cor:phisatp} also follow from the result in
Corollary~\ref{cor:small_level}.  This is however not the case, at least not
directly, since the Siegel $\Phi$-operator involves spaces with actions of
different groups.  In level $1$, we can deduce the surjectivity of the Siegel
operator in positive characteristic from Weissauer's result over $\CC$:
\begin{theorem}\label{thm:surjphi}
Let $g\geq 2$ and $k\geq g+2$ and even, and $\FF$ be a field of characteristic
zero or $p>2g+1$. Then the natural map
\begin{equation*}
    M_k^g(1;\FF)=\H^0\left(\agsat,\osatk\otimes\FF\right)\xrightarrow{\,\Phi=\phisat\,}  \H^0\left(\agsat,\osatk\big|_\triangle\otimes\FF\right)
    =M_k^{g-1}(1;\FF)
\end{equation*}
is surjective.
\end{theorem}

\begin{proof}
  Weissauer proved this result over the field $\CC$ in~\cite[Korollar zum Satz
  8, p. 87]{weissauer}. Through the flat base change $\ZZ\hookrightarrow \CC$ we know that this implies
\begin{equation*}
    \H^0\left(\agsat,\osatk\right)\otimes_\ZZ\CC\xrightarrow{\,\Phi\otimes
      \text{id}\,}  \H^0\left(\agsat,\osatk\big|_\triangle\right)\otimes_\ZZ\CC
\end{equation*}
is surjective. Hence we know that 
\begin{equation*}
 M_k^g(1;\ZZ)=\H^0\left(\agsat,\osatk\right)\xrightarrow{\,\Phi\,}  
 \H^0\left(\agsat,\osatk\big|_\triangle\right)=M_k^{g-1}(1;\ZZ)
\end{equation*}
is a $\ZZ$-linear map of full rank. Suppose it is not surjective, i.e. there is
some $f\in M_k^{g-1}(1;\ZZ)$ such that $f$ is not in the image of $\Phi$.
Since the map has full rank, its cokernel is torsion, so there is a minimal $m\in\NN$ such that $mf$ is in
the image of $\Phi$.  If the field $\FF$ has characteristic zero, then
$m^{-1}\in\FF$ and $\Phi$ is surjective.

It remains to deal with the case where $\FF$ has characteristic $p>2g+1$.
Let $F\in M_k^g(1;\ZZ)$ be such that $\Phi(F)=mf$ for $m\in\NN$ minimal, and let
$\overline{F}\in M_k^g(1;\FF_p)$ be the reduction of $F$ modulo $p$.

Suppose that $p\mid m$.  Then $\Phi(\overline{F})=m\overline{f}=0$, so
$\overline{F}\in \ker\Phi=S_k^g(1;\FF_p)$.  Since $p>2g+1$, by Corollary~\ref{cor:small_level} we know that $\overline{F}$ must lift to a cusp form
$G\in S_k^g(1;\ZZ)$.
Then $(F-G)(q)\equiv 0\mod p$, so all
the Fourier coefficients of $F-G$ are divisible by $p$, therefore $\frac{1}{p}(F-G)\in
M_k^g(1;\ZZ)$.  This gives
\begin{equation*}
\Phi\left(\frac{1}{p}(F-G)\right)=\frac{m}{p}f,
\end{equation*}
contradicting the minimality of $m$.

We conclude that $p$ does not divide $m$. Hence the map
\begin{equation*}
    \H^0\left(\agsat,\osatk\right)\otimes_\ZZ\FF_p\xrightarrow{\,\Phi\otimes
      \text{id}\,}  \H^0\left(\agsat,\osatk\big|_\triangle\right)\otimes_\ZZ\FF_p
\end{equation*} 
is surjective for $p>2g+1$ and Corollary~\ref{cor:small_level} gives the result for 
$\FF_p$. Flat base change $\FF_p\hookrightarrow\FF$ extends it to other fields $\FF$ of characteristic
$p>2g+1$.
\end{proof}

  Note that the condition that $k$ be even (and not just $\geq g+2$) is
  necessary, even over $\CC$: there is a cusp form $\chi$ of level $1$, degree
  $2$ and weight $35\geq 3+2$, but $\chi$ is not in the image of $\Phi$ since
  there are no forms of level $1$, degree $3$ and odd weight. In fact, if $kg$
  is odd then $M_k^g(N;\ZZ)=0$ for $N=1,2$ as $-I\in \Gamma(N)$ implies
  $f=-f$.   However, our results do give some insight into behaviour in level
  $1$ for odd weights. 
  
\begin{theorem}\label{thm:liftC}
  Suppose the characteristic of the base field $\FF$ is zero, or a prime
  $p\geq g(g+1)/2$ not dividing the level $N\geq 3$.  If $k\geq g+2$, there is
  a commutative diagram
  
  \begin{center}
    \begin{tikzpicture}[description/.style={fill=white, inner sep=2pt}]
      \matrix (m) [matrix of math nodes, row sep=3em, column sep=3em,
      text height=1.5ex, text depth=0.5ex]
      { 
      {\H^0\left(\agsat, \omega^{\otimes k}\right)} & 
      {\H^0\left(\agnsat, \omega^{\otimes k}\right)\cong M_k^g(N;\FF)}\\
      {M_k^{g-1}(1;\FF)\cong\H^0\left(\agsat,\omega^{\otimes
      k}\big|_{\Delta_1}\right)} & 
      {\H^0\left(\agnsat,\omega^{\otimes k}\big|_{\Delta_N}\right)}
      \\};
      \path[right hook->, font=\scriptsize]
      (m-1-1) edge node[auto] {} (m-1-2); 
      \path[right hook->, font=\scriptsize]
      (m-2-1) edge node[auto] {$\Psi$} (m-2-2);
      \path[->, font=\scriptsize]
      (m-1-2) edge node[auto] {$\phisat$} (m-2-2);
      \path[->, font=\scriptsize]
      (m-1-1) edge node[auto] {$\Phi$} (m-2-1);
    \end{tikzpicture}
  \end{center}
  
  such that for any $f\in M^{g-1}_k(1;\FF)$ there exists an 
  $F\in M^g_k(N;\FF)$ with $\phisat(F)=\Psi(f)$.
\end{theorem}
\begin{proof}
  Write $\Delta_N=\agnsat-\agn$ for the boundary of the level $N$ Satake
  compactification.  Since the covering morphism $\pi\colon\agnsat\to \agsat$ is
  finite, we have $\oh_{\Delta_1}=\pi_*\oh_{\Delta_N}$.  We also know that
  $\pi^*\omega=\omega$, so by the projection formula
  \begin{equation*}
    \omega^{\otimes k}\big|_{\Delta_1}=
    \omega^{\otimes k}\otimes \pi_*\oh_{\Delta_N}=
    \pi_*\left((\pi^*\omega^{\otimes k})\otimes \oh_{\Delta_N}\right)=
    \pi_*\omega^{\otimes k}\big|_{\Delta_N}.
  \end{equation*}

This means that we have an injection of global sections 
\begin{equation*}
\Psi\colon M_k^{g-1}(1;\FF)\cong 
\H^0\left(\agsat,\omega^{\otimes k}\big|_{\Delta_1}\right)
\hookrightarrow 
\H^0\left(\agnsat,\omega^{\otimes k}\big|_{\Delta_N}\right)
\end{equation*}
which is the missing feature in the commutative diagram in the statement.

Corollary~\ref{cor:phisatp} gives conditions for $\phisat$ to be surjective.
Under these conditions we have that any element embedded by the above map will
have a pre-image under $\phisat$ in $M_k^g(N;\FF)$.
\end{proof}

Over $\CC$ we consider the restriction to just one of the irreducible
components of the cusp of the Satake compactification to highlight the
relevance of this algebraic result in the more analytic setting of the results
of Weissauer~\cite[Korollar zum Satz 8, p. 87]{weissauer}.
\begin{corollary}
  \label{cor:phiN}
Let $k\geq g+2$, $N\geq 3$. Then for any $f\in M^{g-1}_k(1;\CC)$ there exists
an $F\in M^g_k(N;\CC)$ such that
\begin{equation*}
\lim_{t\to\infty} F\begin{pmatrix}it&0\\0&\tau   \end{pmatrix}=f(\tau)
\end{equation*}
for $\tau\in\HH_{g-1}$, $t\in \RR$.
\end{corollary}
(Weissauer's result implies this for \emph{even} weights as any level $1$ form can
be considered a form of level $N\geq 1$. In comparison,
Corollary~\ref{cor:phiN} applies in \emph{both even and odd} weights.)


\bibliographystyle{plain}
\bibliography{base}
\end{document}